\documentclass[authoryear,preprint,12pt]{elsarticle}
\usepackage{amsmath}
\usepackage{amsthm}
\usepackage{amssymb}
\usepackage[plainpages=false,pdfpagelabels]{hyperref}

\newtheorem*{satz}{Theorem}
\newtheorem{prop}{Proposition}

\newcommand{\e}{\ensuremath{\mathbb{E}}}
\newcommand{\var}{\ensuremath{\mathbb{V}}ar}
\newcommand{\id}{1\hspace{-0,9ex}1}
\newcommand{\tr}{\ensuremath{\mathrm{tr}}} 

\date{\today}

\journal{Statistics and Probability Letters}

\begin{document}

\begin{frontmatter}

\title{Semicircle law of Tyler's M-estimator for scatter}

\author[gabi]{Gabriel Frahm}
\ead{gabriel.frahm@hsu-hh.de}

\author[konsti]{Konstantin Glombek\corref{cor}\fnref{dfg}}
\ead{glombek@wiso.uni-koeln.de}

\address[gabi]{Helmut-Schmidt-Universit\"at, Chair for Applied Stochastics, D-22043 Hamburg, Germany}

\address[konsti]{Universit\"at zu K\"oln, Seminar f\"ur Wirtschafts- und Sozialstatistik, Albertus-Magnus-Platz, D-50923 K\"oln, Germany}
\fntext[dfg]{The author is supported by a grant of the German Research Foundation (DFG).}
\cortext[cor]{Corresponding author. Tel. +49 221 4707711}

\begin{abstract}
This paper analyzes the spectral properties of Tyler's M-estimator for scatter $T$. It is shown that if a multivariate sample stems from a generalized spherically distributed population and the sample size $n$ and the dimension $d$ both go to infinity while $d/n\rightarrow0$, then the empirical spectral distribution of $\sqrt{n/d}(T-I)$ converges in probability to the semicircle law, where $I$ is the identity matrix. In contrast to that of the sample covariance matrix, this convergence does not necessarily require the sample vectors to be componentwise independent. Further, moments of the generalized spherical population do not have to exist.
\end{abstract}


\end{frontmatter}

\section{Introduction}

The spectral analysis of large dimensional random matrices has become an active field of research during the last decades because of its broad applicability to many practical problems such as wireless communications, statistics and finance. A main tool of this analysis is the empirical spectral distribution function (ESD) of a $d-$dimensional matrix $A$ having real eigenvalues. Denote by $\lambda_1(A)\leq\ldots\leq\lambda_d(A)$ the eigenvalues of $A$. Then, the ESD of $A$ is defined as $$F^A(x)=\frac{1}{d}\sum_{i=1}^d \id_{(-\infty,x]}(\lambda_i(A)),$$
where $$\id_{M}(x)=\begin{cases}1,&x\in M\\0, &\text{else}\end{cases}$$
for some set $M$. If the entries of $A$ are random variables, then, as $d\rightarrow\infty$ and assuming a certain distribution of these entries, the ESD of $A$ may converge to a non-random limit in some sense. \citet{wigner} investigates special random matrices in the context of quantum mechanics, so-called Wigner matrices, whose expected limiting ESD is the well known semicircle law, a continuous distribution function with density
\begin{equation}
\mathrm{d}w(x)=\frac{1}{2\pi}\sqrt{4-x^2}\id_{[-2,2]}(x) \ \mathrm{d}x \ .
\label{semi}
\end{equation}
In statistics, the limiting behavior of the ESD of large dimensional sample covariance matrices has attained much interest. These matrices are of the type $$S=\frac{1}{n}\sum_{j=1}^nX_jX_j^t,$$ where $X_j=(X_{1j},\ldots,X_{dj})^t,1\leq j\leq n,$ is a centered $d-$dimensional sample. Assuming that $X_{ij}\overset{\text{i.i.d.}}{\sim} (0,1)$, the almost sure limit of the ESD of $S$ under the asymptotics $n,d\rightarrow\infty$ and $d/n\rightarrow y\in(0,\infty)$ is derived in \citet{marcenko}, the famous Mar\v{c}enko-Pastur law. It is given by a distribution function $F_y(x)$ satisfying 
\begin{equation}
\label{mplaw}
\mathrm{d}F_y(x)=\left(1-\frac{1}{y}\right)^+ \ \mathrm{d}\delta_0(x)+f_y(x) \ \mathrm{d}x,
\end{equation}
where
$$f_y(x)=\frac{1}{2\pi xy}\sqrt{(a_+-x)(x-a_-)}\id_{[a_-,a_+]}(x),$$
$a_\pm=(1\pm\sqrt{y})^2$ and $\delta_0$ is the Dirac delta function in $0$. Further, a suitable standardization of $S$ yields the semicircle law as $n,d\rightarrow\infty, d/n\rightarrow y=0$. More precisely, \citet{bai2} show that if $X_{ij}\overset{\text{i.i.d.}}{\sim} (0,1)$ and $\e|X_{11}|^4<\infty$, then, as $n,d\rightarrow\infty$ and $d/n\rightarrow 0$, the ESD of
$$S^*:=\sqrt{\frac{n}{d}}(S-I),$$
converges almost surely to the semicircle law (\ref{semi}). Here, $I$ denotes the identity matrix.\par
There is research about the question of how the strict assumption of an identical and independent distribution of the $X_{ij}$ can be weakened. \citet{bai3} give an overview that a Lindeberg-type condition for the $X_{ij}$ can replace an identical distribution. Further, they show that the convergence of the ESD of $S^*$ to the semicircle law is still valid if $S^*$ is sparse which means that some entries of $S^*$ are missing in a certain sense. But still the assumption of horizontal and vertical independence of the $X_{ij}$ is required. We show in this paper that if we do not use the sample covariance matrix but \textit{Tyler's M-estimator} to estimate the true covariance matrix of a standard normal population and standardize this estimator in the same manner as $S$, then the associated ESD also converges almost surely to the semicircle law (\ref{semi}). Moreover, it is shown that this convergence holds even if the population is \textit{generalized spherically} distributed. Since the components of a generalized spherical population are not independent (except in the case of normality), the condition of vertical independence of the $X_{ij}$ is weakened. We see further that certain moment conditions can be relaxed as well.\par
The next section briefly provides all necessary information about shapes matrices and their connection to Tyler's M-estimator so that the main results can be stated. Section \ref{proofs} will then give the proof of the results which is followed by a small outlook for future research.

\section{Tyler's M-estimator}

The shape matrix of a $d-$dimensional population $X$ is defined as the symmetric and positive definite solution $\Omega=\Omega(X)\in\mathbb{R}^{d\times d}$ of the equation
\begin{equation}
\label{shape}
\e\left(\frac{\Omega^{-1/2}(X-\mu)(X-\mu)^t\Omega^{-1/2}}{(X-\mu)^t\Omega^{-1}(X-\mu)}\right)=\frac{1}{d}I,
\end{equation}
where $\mu\in\mathbb{R}^d$ is the center of the distribution of $X$ and $\Omega^{-1/2}$ denotes the symmetric root of the inverse $\Omega^{-1}$. Note that $\mu$ does not have to be the expectation of $X$ but may also be its median (see also \citet{frahm2}, Section 3). In the following, only $\mu=0$ will be of interest so that we will not further discuss this point. Equation (\ref{shape}) determines $\Omega$ uniquely up to a scalar multiple so that a suitable scaling is needed. We choose $\tr(\Omega)=d$, where $\tr$ denotes the trace operator, and refer to \citet{frahm3} for a detailed discussion about shape matrices and their scales.\par
If $X$ is spherically distributed, i.e.,
\begin{equation}
\label{sphere}
X\overset{\mathrm{d}}{=}RU^{(d)},
\end{equation}
where $R\geq0$ is a scalar random variable being independent from $U^{(d)}\sim\mathcal{U}(S^{d-1})$, where $\mathcal{U}(S^{d-1})$ denotes the uniform distribution over the unit sphere in $\mathbb{R}^d$, then $\Omega(X)=\Omega(U^{(d)})=I$ with respect to the center $\mu=0$. We see that the shape matrix does not depend on $R$. Especially, if $R\overset{\mathrm{d}}{=}\sqrt{\chi^2_d}$, where $\chi_d^2$ denotes a random variable which is $\chi^2$ distributed with $d$ degrees of freedom, we obtain that $X\sim N_d(0,I)$, i.e., $X$ is standard normal so that the covariance and shape matrix of $X$ agree.\par 
Now, we assume that $\mu$ is known and set w.l.o.g. $\mu=0$. The shape matrix of $X$ can be estimated by its sample counterpart $\hat{\Omega}$ as the solution of
$$\frac{1}{n}\sum_{j=1}^n\frac{\hat{\Omega}^{-1/2}X_jX_j^t\hat{\Omega}^{-1/2}}{X_j^t{\hat{\Omega}^{-1}}X_j}=\frac{1}{d}I$$
or, equivalently, of
\begin{equation}
\label{eq2}
\hat{\Omega}=\frac{d}{n}\sum_{j=1}^n\frac{X_jX_j^t}{X_j^t\hat{\Omega}^{-1}X_j},
\end{equation}
where $X_1,\ldots,X_n$ is a sample drawn from $X$. The estimator $\hat{\Omega}$ is known as Tyler's M-estimator for scatter and is introduced by \citet{tyler}, which is why we set $T=\hat{\Omega}$ for short. Again, we choose the scaling $\tr(T)=d$ in order to uniquely define $T$ as it is also proposed by \cite{tyler}. The existence of $T$ is assured if $n\geq d$ (see \citet{tyler2}, Lemma 2.1.) and its computation can be done by performing an iteration scheme leading to an unique symmetric and positive definite solution of Equation (\ref{eq2}) (see also \citet{tyler}).\par
Since $\Omega(X)=\Omega(RX)$ for a scalar random variable $R$, this invariance is inherited to $T=T(X)$ meaning that $T(X)=T(Y)$ if $X\overset{\mathrm{d}}{=}RY$ for a $d-$dimensional random variable $Y$. This property is quite appealing for elliptical populations (with center $0$), i.e., $$X\overset{\mathrm{d}}{=}R\Lambda U^{(k)},$$ where $\Lambda\in\mathbb{R}^{d\times k}$ and $R\geq0$ is a scalar random variable being independent from $U^{(k)}\sim\mathcal{U}(S^{k-1})$ (see also \citet{fang}, Chapter 2). Hence, $T$ is distribution-free within the class of elliptical distributions.\par
\cite{frahm_phd} introduces the class of \textit{generalized elliptical distributions} whose members have the same stochastic representation as an elliptical random variable. But this class additionally allows for $R<0$ and dependence between $R$ and $U^{(k)}$. These features are quite useful when dealing with financial data as it is mentioned in \citet{frahm2}. It is clear that $T$ is also distribution-free within that class. \citet{frahm_phd}, Chapter 4, derives $T$ as a maximum likelihood estimator for $\Omega(X)$ assuming that $X$ is generalized elliptically distributed. Thus, $T$ has many desired properties such as consistency and asymptotic normality.\par
Tyler's M-estimator has another outstanding property concerning robustness. \citet{tyler} shows that $T$ is the ``most robust'' estimator for the shape matrix of an elliptical population. This means that if $X$ is elliptical, then the maximum asymptotic variance of $T$ is a minimum within the set of maximum asymptotic variances of all consistent and asymptotically normally distributed shape matrix estimators.\par  
Following the idea of generalized elliptical distributions, we define the class of \textit{generalized spherical distributions} as the set of all random variables $X$ having the stochastic representation (\ref{sphere}), where $R\not\equiv0$ is a scalar random variable and $U^{(d)}\sim\mathcal{U}(S^{d-1})$. Similar to generalized elliptical distributions, $R$ and $U^{(d)}$ may depend on each other and $R$ may also take negative values. Clearly, we have $\Omega(X)=I$ with respect to the center $\mu=0$ if $X$ is generalized spherically distributed. This class of distributions is of interest in the following theorem.

\begin{satz}
Let $X_1,\ldots,X_n$ be an i.i.d. sample drawn from a generalized spherical population $X$ of dimension $d$. Let $T$ be Tyler's M-estimator being normalized so that $\tr(T)=d$. Then, as $n,d\rightarrow\infty$ and $d/n\rightarrow0$, 
\begin{enumerate}
	\item the ESD of $$T^*:=\sqrt{\frac{n}{d}}\left(T-I\right)$$
converges in probability to the semicircle law (\ref{semi}) and
 \item $\lambda_1(T^*)\overset{P}{\longrightarrow}-2,\lambda_d(T^*)\overset{P}{\longrightarrow}2 \ .$
\end{enumerate}
\end{satz}

Here, ``$\overset{P}{\longrightarrow}$'' denotes convergence in probability. Some consequences of the theorem are worth to point out. First, in contrast to the convergence of the ESD of $S^*$ to the semicircle law, it is not assumed that the components of $X$ are independent. The uncorrelatedness of the components of $U$ is required instead because of the distribution-freeness of $T$. Further, moments of $X$ do not have to exist. For example, if $R$ and $U$ are independent and $R\overset{\text{d}}{=}\sqrt{dF_{d,p}}$, where $F_{d,p}$ is a $F-$distributed random variable with $d$ and $p$ degrees of freedom, the population has a $d-$dimensional t-distribution with $p$ degrees of freedom. In the case of $p=1$, we obtain the Cauchy distribution whose expectation does not exist. This is a sharp contrast to the sample covariance matrix which even requires the existence of the fourth moment of the components of $X$ (see \citet{bai2}).

\section{Proof of the theorem}
\label{proofs}

In the following, we consider $n$ as a integer-valued function of $d$ with $\lim_{d\rightarrow\infty}n(d)=\infty$ and $d=o(n)$ as $d\rightarrow\infty$. So, we just write $d\rightarrow\infty$ for $d,n\rightarrow\infty$ and $d/n\rightarrow0$. Almost sure convergence will be denoted by $\overset{\text{a.s.}}{\rightarrow}$ and $\overset{\mathcal{L}^1}{\rightarrow}$ stands for convergence in mean. Further, we set $\|A\|_2:=\max\{|\lambda_1(A)|,|\lambda_d(A)|\}$ for a symmetric $d-$dimensional matrix $A$ which may be deterministic or at random. In the latter case, $\|A\|_2$ is also a random variable.\par
First, because of the distribution-freeness of $T$, we may assume w.l.o.g. that $X\sim N_d(0,I)$, i.e., $R\overset{\text{d}}{=}\sqrt{\chi^2_d}$ being independent from $U^{(d)}$. We will now prove the first part of the theorem by applying the moment convergence theorem (MCT) with respect to that population. The $m-$th moment of the ESD of $T^*$ is given by
$$\int x^m \ \mathrm{d}F^{T^*}(x)=\frac{1}{d}\sum_{i=1}^d \lambda_i^m(T^*)=\frac{1}{d}\tr\left([T^*]^m\right) \ .$$
Since the semicircle law is uniquely defined by its moments (because its support is compact), the MCT is applicable. It says that the convergence of
$$\frac{1}{d}\tr\left([T^*]^m\right)\overset{P}{\longrightarrow} \int x^m \ \mathrm{d}w(x)$$
for every fixed $m\in\mathbb{N}$ as $d\rightarrow\infty$ is sufficient for
$$F^{T^*}(x)\overset{P}{\longrightarrow} w(x)$$
as $d\rightarrow\infty$ (see also the introduction in \citet{bai}). We will show that 
\begin{equation}
\label{nr1}
\left|\frac{1}{d}\tr\left([T^*]^m\right)-\frac{1}{d}\tr\left([S^*]^m\right)\right|\overset{P}{\longrightarrow} 0
\end{equation}
as $d\rightarrow\infty$. Since we have from \citet{bai2} that
\begin{equation*}
\frac{1}{d}\tr\left([S^*]^m\right)\overset{P}{\longrightarrow}\int x^m \ \mathrm{d}w(x)
\end{equation*}
as $d\rightarrow\infty$, the triangle inequality leads to the result.  Now, we need four small propositions.

\begin{prop}
\label{prop1b}
Let $X\sim N_d(0,I)$. Then:
$$\|T^*-S^*\|_2\overset{P}{\longrightarrow}0 \text{ as }d\rightarrow\infty$$
\end{prop}

\begin{proof}
\citet{dumbgen} shows in Theorem 5.4.:
$$\e\|T-S\|_2=o\left(\sqrt{d/n}\right) \text{ as }d\rightarrow\infty$$
It follows that
$$\e\|T^*-S^*\|_2=\sqrt{n/d} \ \e\|T-S\|_2=o(1) \text{ as }d\rightarrow\infty$$
which means
$$\|T^*-S^*\|_2\overset{\mathcal{L}^1}{\longrightarrow}0  \text{ as }d\rightarrow\infty$$
which implies the assertion.
\end{proof}

Next, we have:

\begin{prop}
\label{prop2}
Let $A,B$ be $d-$dimensional symmetric matrices. Then: $$\forall_{1\leq i\leq d}:|\lambda_i(A)-\lambda_i(B)|\leq\|A-B\|_2$$
\end{prop}

\begin{proof}
The generalized Weyl inequality says that $\lambda_i(A)+\lambda_{d+1-i}(B)\leq \lambda_d(A+B)$ (see, e.g., \citet{schott}, Theorem 3.23.). Then, we have either
\begin{align*}
0&\leq \lambda_i(A)-\lambda_i(B)=\lambda_i(A)+\lambda_{d+1-i}(-B)\leq \lambda_d(A-B)\leq\|A-B\|_2\\
&\text{or}\\
0&\leq \lambda_i(B)-\lambda_i(A)=\lambda_i(B)+\lambda_{d+1-i}(-A)\leq \lambda_d(B-A)\\
&=-\lambda_1(A-B)\leq\|A-B\|_2 \ .
\end{align*}
\end{proof}

\begin{prop}
\label{prop3}
Let $X\sim N_d(0,I)$. Then, we have that $$\|S^*\|_2\overset{\text{a.s.}}{\longrightarrow}2  \text{ as }d\rightarrow\infty \ .$$
\end{prop}

\begin{proof}
From \citet{dette}, Corollary 2.2., it follows that $\lambda_1(S^*)\overset{\text{a.s.}}{\longrightarrow}-2$ and $\lambda_d(S^*)\overset{\text{a.s.}}{\longrightarrow}2$ as $d\rightarrow\infty$.
\end{proof}

Now, define the interval $$B_i:=\left\{\alpha\lambda_i(T^*)+(1-\alpha)\lambda_i(S^*) \  \Big\vert \ \alpha\in[0,1]\right\} \ .$$

\begin{prop}
\label{prop4}
Again, let $X\sim N_d(0,I)$. Then, we have that
$$\sup_{\lambda\in\bigcup_{i=1}^dB_i}|\lambda|\overset{P}{\longrightarrow}2 \text{ as }d\rightarrow\infty \ .$$
\end{prop}

\begin{proof}
From the definition of $B_i$, it holds that
\begin{align*}
\forall_{\lambda\in\bigcup_{i=1}^dB_i}\exists_{1\leq j\leq d,\alpha\in[0,1]}:\lambda&=\alpha\lambda_j(T^*)+(1-\alpha)\lambda_j(S^*)\\
&=\alpha\left(\lambda_j(T^*)-\lambda_j(S^*)\right)+\lambda_j(S^*)
\end{align*}
from which follows:
\begin{align*}
|\lambda|&\leq\alpha\left|\lambda_j(T^*)-\lambda_j(S^*)\right|+|\lambda_j(S^*)|\\
&\leq\alpha\|T^*-S^*\|_2+\|S^*\|_2\overset{P}{\longrightarrow}2
\end{align*}
as $d\rightarrow\infty$ using Propositions \ref{prop1b}, \ref{prop2} and \ref{prop3}.
\end{proof}

Now, we can estimate (\ref{nr1}) as follows:
\begin{align*}
\left|\frac{1}{d}\tr\left([T^*]^m\right)-\frac{1}{d}\tr\left([S^*]^m\right)\right|&\leq\frac{1}{d}\sum_{i=1}^d\left|\lambda_i^m(T^*)-\lambda_i^m(S^*)\right|\\
&\overset{(*)}{\leq} \frac{1}{d}\sum_{i=1}^dm\sup_{\lambda\in B_i}\left|\lambda^{m-1}\right|\underbrace{\left|\lambda_i(T^*)-\lambda_i(S^*)\right|}_{\leq\|T^*-S^*\|_2\text{ (Prop.  \ref{prop2})}}\\
&\leq m \ \|T^*-S^*\|_2 \frac{1}{d}\sum_{i=1}^d\underbrace{\sup_{\lambda\in B_i}|\lambda|^{m-1}}_{\leq\sup_{\lambda\in\bigcup_{i=1}^dB_i}|\lambda|^{m-1}}\\
&\leq m \underbrace{\|T^*-S^*\|_2}_{\overset{P}{\longrightarrow}0\text{ (Prop. \ref{prop1b})}} \underbrace{\left(\sup_{\lambda\in\bigcup_{i=1}^dB_i}|\lambda|\right)^{m-1}}_{\overset{P}{\longrightarrow}2^{m-1}\text{ (Prop. \ref{prop4})}}\overset{P}{\longrightarrow} 0
\end{align*}
as $d\rightarrow\infty$. The inequality $(*)$ is due to the mean value theorem for continuously differentiable functions. All in all, the convergence in (\ref{nr1}) is shown, which completes the proof of the first part of the theorem. The second part of the theorem is a simple consequence of the preceding results. We have that
\begin{align*}
|\lambda_1(T^*)+2|&\leq|\lambda_1(T^*)-\lambda_1(S^*)|+|\lambda_1(S^*)+2|\\
&\leq\|T^*-S^*\|_2+|\lambda_1(S^*)+2|\overset{P}{\longrightarrow}0,\\
|\lambda_d(T^*)-2|&\leq|\lambda_d(T^*)-\lambda_d(S^*)|+|\lambda_d(S^*)-2|\\
&\leq\|T^*-S^*\|_2+|\lambda_d(S^*)-2|\overset{P}{\longrightarrow}0
\end{align*}
as $d\rightarrow\infty$ using Propositions \ref{prop1b} and \ref{prop2} and Corollary 2.2. from \citet{dette}, which completes the proof of the theorem. 

\section{Outlook for future research}

\citet{yin2} show that, as $n,d\rightarrow\infty$ and $d/n\rightarrow y\in(0,1)$, the ESD of $S$ converges in probability to a non-random limit if the population is spherical. This limiting distribution is described by its moments and is unequal to the Mar\v{c}enko-Pastur law unless the population is standard normal. In contrast, \citet{frahm}, Section 3.2., give evidence that the limiting ESD of $T$ equals the Mar\v{c}enko-Pastur law (\ref{mplaw}) under these asymptotics if the population is generalized spherical. Regarding the proof in Section \ref{proofs}, this conjecture will be shown if one proves that $\|T-S\|_2\overset{P}{\rightarrow}0$ for a standard normal population as $n,d\rightarrow\infty$ and $d/n\rightarrow y\in(0,1)$. Note that considering $y>1$ is not possible because $T$ does not exist for $n<d$. An analysis of the proof of Theorem 5.4. in \cite{dumbgen} may provide a solution to this problem.\par
Another concern is the establishment of the almost sure convergence of the ESD of $T$ and $T^*$. Here, we need additional results on the second moments of $\|T-S\|_2$ and $\|T^*-S^*\|_2$. For example, if one could show that $\var(\|T^*-S^*\|_2)=\mathcal{O}(d^{-(1+\delta)})$ as $d\rightarrow\infty$ for some $\delta>0$, then the almost sure convergence of the ESD of $T^*$ to the semicircle law would follow by the Borel-Cantelli lemma.

\section*{Acknowledgements}

The authors are very grateful to Lutz D\"umbgen, Karl Mosler and David E. Tyler for their helpful comments and suggestions.

\section*{References}
\bibliographystyle{model2-names}
\bibliography{bibdata2}
\addcontentsline{toc}{section}{References}

\end{document}